\documentclass[11pt]{article}
\usepackage{amsthm,amsmath,amssymb,nccmath}
\usepackage{dsfont}
\usepackage{graphicx}
\usepackage{hyperref}
\usepackage{tikz}
\usepackage{cases}
\usepackage{todonotes}

\newtheorem{Thm}{Theorem}
\newtheorem{Lem}[Thm]{Lemma}

\newtheorem{Pro}[Thm]{Proposition}

\newtheorem{Rem}[Thm]{Remark}

\newcommand{\R}{\mathbb{R}}

\usepackage[normalem]{ulem}
\definecolor{DarkGreen}{rgb}{0,0.5,0.1} 

\newcommand\soutD{\bgroup\markoverwith
{\textcolor{DarkGreen}{\rule[.5ex]{2pt}{1pt}}}\ULon}
\newcommand{\Hm}[1]{\leavevmode{\marginpar{\tiny%
$\hbox to 0mm{\hspace*{-0.5mm}$\leftarrow$\hss}%
\vcenter{\vrule depth 0.1mm height 0.1mm width \the\marginparwidth}%
\hbox to
0mm{\hss$\rightarrow$\hspace*{-0.5mm}}$\\\relax\raggedright #1}}}

\title{A note on the failure of the Faber-Krahn  inequality for the vector Laplacian}
\author{David Krej\v{c}i\v{r}\'{\i}k,$^a$ 
Pier Domenico Lamberti\,$^b$  
and  Michele Zaccaron\,$^c$ }

\date{\small 
\vspace{-5ex}
\begin{quote}
\emph{
\begin{itemize}
\item[$a)$] 
Department of Mathematics, Faculty of Nuclear Sciences and 
Physical Engineering, Czech Technical University in Prague, 
Trojanova 13, 12000 Prague 2, Czech Republic;
david.krejcirik@fjfi.cvut.cz.%
\item[$b)$] 
Dipartimento di Tecnica e Gestione dei Sistemi Industriali (DTG),
University of Padova,
Stradella S. Nicola 3, 36100 Vicenza, Italy;
pierdomenico.lamberti@unipd.it
\item[$c)$] 
Institut Fresnel, 
Facult\'e des Sciences - Campus Saint J\'er\^ome, 
Avenue Escadrille Normandie-Ni\'emen, 
13397 Marseille CEDEX, France;
zaccaron@fresnel.fr
\end{itemize}
}
\end{quote}
October 1, 2024
}

\begin{document}

\maketitle

\abstract{We consider a natural eigenvalue problem for the vector Laplacian related to stationary Maxwell's equations in a cavity and we prove that an analog of the celebrated Faber-Krahn inequality doesn't hold,  regardless of whether a volume or perimeter constraint is applied.}

\section{Introduction}

In this paper  we consider  the  eigenvalue problem for the $\operatorname{curl}\operatorname{curl}$ operator 
\begin{equation} \label{el:maxwell:pb}
\begin{cases}
\operatorname{curl}\operatorname{curl}u=\lambda u &\text{in }\Omega,\\
\operatorname{div}u=0 &\text{in }\Omega,\\
u \times \nu =0 &\text{on }\partial\Omega, \\
\end{cases}
\end{equation}
on bounded domains $\Omega$ in $\R^3$ with Lipschitz boundaries, in the unknown vector field $u~:~\Omega \to \R^3$ (the eigenvector) and the unknown $\lambda \in {\mathbb{R}}$ (the eigenvalue).
Here $\nu$ denotes the unit outer normal to $\partial\Omega$. 
Recall that $
\operatorname{curl}\operatorname{curl}u =-\Delta u +\nabla \operatorname{div}u $ hence $
\operatorname{curl}\operatorname{curl}u =-\Delta u$ if $u$ is a divergence-free vector field. Note also that the second condition in \eqref{el:maxwell:pb} is immediately implied by the first equation if $\lambda \ne 0$. 
Problem \eqref{el:maxwell:pb} admits  a sequence of eigenvalues  that are ordered in increasing order, taking into account their multiplicity, as follows: 
\begin{equation*}
0 \leq \lambda_1^\Omega \leq \lambda_2^\Omega \leq \dots \lambda_j^\Omega \leq \dots \nearrow +\infty.
\end{equation*}
We refer to \cite{lamzac21} for  basic results concerning problem \eqref{el:maxwell:pb} and  for references.

We assume that  the boundary of  $\Omega$ has only one connected component so that  zero is not an eigenvalue and $\lambda_1^{\Omega}>0$. 
Under these assumptions on $\Omega$,  
we consider the problem of minimizing and maximizing $\lambda_1^{\Omega}$ under  the volume constraint $|\Omega|={\rm constant}$ or the perimeter constraint $|\partial \Omega|={\rm constant}$ and we prove that 
\begin{equation}\label{inf}
\inf_{|\Omega|=\text{const.}} \lambda_1^\Omega =0,\ \ \inf_{  |\partial \Omega |=\text{const.}} \lambda_1^\Omega =0\end{equation}
and 
\begin{equation}\label{sup}\sup_{|\Omega|=\text{const.}} \lambda_1^\Omega =+\infty, \ \ \sup_{ |\partial \Omega |=\text{const.}} \lambda_1^\Omega =+\infty.\end{equation}
Note that $|\Omega|$ denotes the  volume of $\Omega$, while $|\partial \Omega|$ the two-dimensional surface measure of $\partial\Omega$.

The first equality in  \eqref{inf} and the equalities in \eqref{sup} will be  easily proved by considering suitable families of cuboids, i.e., rectangular parallelepipeds,   in which case explicit formulas are known. By using the same formulas,  it turns our that 
\begin{equation}\label{lower}
\inf_{\substack{\Omega\ \text{cuboid} \\ |\partial\Omega|=k}} \lambda_1^\Omega =\frac{4 \pi^2}{k},
\end{equation}
for all $k>0$, see Theorem~\ref{theo:per}.  Thus,  the second equality in  \eqref{inf}  will be  proved by considering another family of non-convex domains, namely suitable three dimensional  dumbbell domains: in this case, we shall use  the full description of the spectrum of problem  \eqref{el:maxwell:pb} in cross product domains provided by \cite{cosdau}.

It is interesting to observe that if $\Omega$ is a ball with surface area equal to $k$ then its first eigenvalue is larger than the lower bound in \eqref{lower}, see Remark~\ref{ball}.  The reader interested in explicit computations of Maxwell's eigenvalues  for analogous boundary value problems can find more results and formulas in \cite{ferlamstra}.

The main  physical motivation for studying problem \eqref{el:maxwell:pb} comes from the analysis of  the stationary Maxwell's equations,  in which case $u$ plays the role of an electric field in a cavity $\Omega$ surrounded by a perfect conductor $\partial\Omega$ that is responsible for  the boundary condition. In particular, the study of electromagnetic cavities has applications in designing cavity resonators or shielding structures for electronic circuits.   We refer to  \cite[Chp.~10]{hanyak} for a detailed  introduction to this subject and to the  extensive monographs \cite{kihe, rsy}  on the mathematical theory of electromagnetism.  

We note that the energy space naturally associated with problem \eqref{el:maxwell:pb} is $X_N(\operatorname{div} 0, \Omega):=H_0(\operatorname{curl}, \Omega )\cap H( \operatorname{div} 0, \Omega)$ that is defined as the space of divergence free vector fields $u$ in $L^2(\Omega)^3$ with $\operatorname{curl} u$ in $L^2(\Omega)^3$, satisfying the boundary condition in  \eqref{el:maxwell:pb} (as usual, the boundary condition is understood 
in the weak sense by defining $H_0(\operatorname{curl}, \Omega )$ as the closure in $H(\operatorname{curl}, \Omega )$  of the space of smooth vector fields with compact support). 
In particular it turns out that 
$$
 \lambda_1^\Omega =\min_{\substack{u\in X_N(\operatorname{div} 0, \Omega)\\ u\ne 0}} \frac{\int_{\Omega}|\operatorname{curl} u |^2dx }{\int_{\Omega}|u|^2dx}\, .  $$

It is  clear that the eigenvalue problem  \eqref{el:maxwell:pb} is the natural vectorial version of the eigenvalue problem for the Dirichlet Laplacian 
in which case the energy space is the usual Sobolev space $H^1_0(\Omega)$. For the Dirichlet Laplacian, the celebrated Faber-Krahn inequality states that the ball minimizes the first eigenvalue under volume constraint, see e.g., \cite{henrot}. Thus, our observations  point out that, although the  ball is critical for the elementary symmetric functions of the eigenvalues of  \eqref{el:maxwell:pb} under both volume and perimeter constraint (cf. \cite{lamzac21}), the Faber-Krahn inequality doesn't hold, no matter whether the volume or the perimeter constraint is used. On the other hand the lower bound in \eqref{lower} suggests that the problem of minimizing the first eigenvalue under perimeter constraint and additional geometric constraints could be of interest. With regard to this, we observe that it has been proved in \cite{pauly} that in case of bounded convex domains the first non-zero eigenvalue $\mu_{1,\mathcal{N}}^{\Omega}$ of the Neumann Laplacian in $\Omega$ satisfies the estimate 
$\mu_{1,\mathcal{N}}^{\Omega} \le \lambda_1^{\Omega}$, and a classical inequality by   Payne and Weinberger~\cite{payne} states that if $\Omega$ is convex then  $\mu_{1,\mathcal{N}}^{\Omega} \geq \pi^2/d^2$ where $d$ is the diameter of $\Omega$. See also \cite{roh} for further references  and  inequalities between Maxwell's eigenvalues and the eigenvalues of the Dirichlet Laplacian. 

We observe  that $\lambda_1^{\Omega}$ coincides with the first eigenvalue $\mu_1^{[2]}$ of the Hodge Laplacian acting on $2$-forms with so-called absolute boundary conditions, see  \cite{savo} for notation. By using this non-trivial fact, it is possible to give another proof  of the first equality in  \eqref{inf} and the equalities in \eqref{sup} based the general two-sided estimate proved for convex domains in \cite[Thm.~1.1]{savo}. We explain how to do that in  Remark~\ref{savorem}.  As far as this approach is concerned, we note that the first equalities in  \eqref{inf} and \eqref{sup} for the limiting behavior of $\mu_1^{[2]}$ to zero or infinity under volume constraint, can also be found in \cite[\S~5]{guesav}.

Finally, we note that the failure of the Faber-Krahn inequality for a different vectorial eigenvalue problem associated with the Stokes operator has been recently discussed in \cite{henmazpri}.

\section{Cuboids}

In this section we consider cuboids  $\Omega$ in $\R^3$ with  sides of length $\ell_1,\ell_2,\ell_3$ that are always assumed to be  in the order $\ell_1\geq\ell_2 \geq \ell_3$. Although it is not necessary, the reader may think of using a system of coordinates that allows to represent $\Omega$ in the form $\Omega=(0,l_1)\times (0,l_2)\times (0,l_3)$. 
It is well-known  that 
\begin{equation*}
\lambda_1^{\Omega} = \pi^2 \left( \frac{1}{\ell_1^2} + \frac{1}{\ell_2^2} \right),
\end{equation*}
see e.g.,  \cite{cosdau}.
This formula allows to give an immediate proof of the following theorem concerning the case of the volume constraint. 

\begin{Thm}  For any fixed $k>0$ we have
$$\inf_{\substack{\Omega \text{\ \rm  cuboid} \\ |\Omega|=k}} \lambda_1^\Omega=0, \qquad  \sup_{\substack{\Omega \text{\ \rm  cuboid} \\ |\Omega|=k}} \lambda_1^\Omega=+\infty .$$
\end{Thm}

\begin{proof} 
We begin by considering  cuboids $\Omega$ with $|\Omega|=1$. 
Fix $0<\ell\leq 1$ and take $\ell_3=\ell$ and $\ell_1=\ell_2=\ell^{-1/2}$. The volume of such a  cuboid is equal to $1$ and the first Maxwell eigenvalue satisfies
\begin{equation}\label{zero}
\lambda_1^{\Omega}=\pi^2\left( \frac{1}{\ell_1^2} +\frac{1}{\ell_2^2}\right)= 2 \pi^2 \ell \xrightarrow[\ell \to 0^+]{}0^+.
\end{equation}
If instead we take  $\ell_2=\ell_3 = \ell $ and $\ell_1=\ell^{-2}$, then we have again that $|\Omega|=1$  and 
\begin{equation}\label{infinito}
\lambda_1^{\Omega}=\pi^2\left( \frac{1}{\ell_1^2} +\frac{1}{\ell_2^2}\right)= \pi^2 \left(\ell^4 + \frac{1}{\ell^2} \right) \xrightarrow[\ell \to 0^+]{} +\infty.
\end{equation}
Thus the theorem is proved for  cuboids satisfying $|\Omega|=1$.   The general case of arbitrary fixed volumes, can be proved by observing  that  
$\lambda_1^{\alpha\Omega}=\lambda_1^\Omega/ \alpha^2 $ for any $\alpha >0$. 
\end{proof}

As far as the perimeter constraint is concerned, 
we can prove the following theorem.

\begin{Thm} \label{theo:per}
For any fixed $k>0$ we have
\begin{equation}\label{cubper}
\inf_{\substack{\Omega \text{\ \rm   cuboid} \\ |\partial\Omega|=k }} \lambda_1^\Omega= \frac{4 \pi^2}{k},\ \ \sup_{\substack{\Omega \text{\ \rm  cuboid}  \\ |\partial\Omega|=k}} \lambda_1^\Omega=+\infty
\end{equation}
and the infimum is not attained. 
\end{Thm}
\begin{proof} 
We begin with proving the second equality in \eqref{cubper}. Possibly rescaling,  it is enough to consider the  case $|\partial \Omega|=2$ which means that  the lengths  $\ell_1,\ell_2,\ell_3$ of the sides of the cuboid have to  satisfy the following constraint
\begin{equation*}
\ell_1 \ell_2 + \ell_1 \ell_3 + \ell_2 \ell_3=1.
\end{equation*}
Consider $0< \ell \leq 1/\sqrt{3}$ and  take $\ell_2=\ell_3=\ell$ and $\ell_1 = (1-\ell^2)/2 \ell$. Note that the perimeter constraint is satisfied and that  the assumptions on $\ell$ guarantee that $\ell_1\geq \ell$. It follows that 
\begin{equation}\label{infinitodue}
\lambda_1^{\Omega}= \pi^2 \left( \frac{4\ell^2}{(1-\ell^2)^2} + \frac{1}{\ell^2} \right) \xrightarrow[\ell \to 0^+]{} +\infty
\end{equation}
and the second equality in \eqref{cubper} is proved. 

We now prove the first equality in \eqref{cubper}.   By the perimeter constraint $|\partial \Omega|=k$,  the lengths  $\ell_1,\ell_2,\ell_3$ of the sides of the cuboid have to  satisfy the following constraint
$$
|\partial \Omega |= 2(\ell_1 \ell_2 + \ell_1 \ell_3 + \ell_2 \ell_3)=k
$$
which implies that 
\begin{equation}\label{elledue}
\ell_2=\frac{\frac{k}{2}-\ell_3 \ell_1}{\ell_3+\ell_1}\, .
\end{equation}
Note that assuming that  $\ell_3\le \ell_2\le  \ell_1$ is equivalent to requiring that
\begin{equation}\label{kappasesti}
\sqrt{\ell_3^2+\frac{k}{2}}-\ell_3\le \ell_1\le \frac{\frac{k}{2}-\ell_3^2}{2\ell_3},\  {\rm with}\ 0<\ell_3\le \sqrt{k/6}.
\end{equation}
 By Young's inequality we get
\begin{equation}\label{elledue1}
\lambda_1^{\Omega}=\pi^2\left( \frac{1}{\ell_1^2} +\frac{1}{\ell_2^2}\right) =   \pi^2\left(   \frac{1}{\ell_1^2} + \frac{(\ell_3+\ell_1)^2}{(\frac{k}{2}-\ell_3 \ell_1)^2}\right)
 \geq   \frac{2(\ell_3+\ell_1)\pi^2}{\ell_1\left(\frac{k}{2}- \ell_3 \ell_1\right)  }> \frac{4\pi^2}{k},
\end{equation}
which proves that 
$$
\inf_{\substack{\Omega \text{\ \rm   cuboid}  \\ |\partial\Omega|=k }} \lambda_1^\Omega\geq \frac{4 \pi^2}{k}
$$
and that the infimum cannot be a minimum. Choosing $\ell_1 = \sqrt{\ell_3^2+\frac{k}{2}}-\ell_3$ and passing to the limit as $\ell_3\to 0^+$ in the second equality of \eqref{elledue1} yield
$$
\lambda_1^{\Omega}  \xrightarrow[\ell_3 \to 0^+]{}  \pi^2\left(\frac{1}{\ell_1^2}+\frac{4\ell_1^2}{k^2}\right)= \frac{4 \pi^2}{k}
$$
and the proof of the first equality in the statement is proved. 
\end{proof}

\begin{Rem} The proof of the previous theorem shows that the infimum in \eqref{cubper} is reached by a minimizing  sequence of cuboids that degenerate to a square with side $\sqrt{k/2}$. 
\end{Rem}

\begin{Rem}\label{savorem}
Following the terminology and the notation in \cite{savo}, consider the Hodge Laplacian acting on $p$-forms with absolute boundary conditions on a bounded convex domain $\Omega$ of $\mathbb{R}^3$,  and denote by  $\mu_1^{[p]}$ the corresponding first eigenvalue. Here $p=0,1,2,3$ and $\mu_1^{[0]}\le \mu_1^{[1]}\le \mu_1^{[2]}\le \mu_1^{[3]}$. Recall from \cite{savo} that $\mu_1^{[0]}=\mu_1^{[1]}$ is the first positive  eigenvalue of the Neumann Laplacian, while $\mu_1^{[3]}$ is the first eigenvalue of the Dirichlet Laplacian. Denote by 
$\mu_1^{[p]'}$ the first eigenvalue of  the Hodge Laplacian acting on exact $p$-forms. Keeping in mind the standard identification of vector fields $u$ in ${\mathbb{R}}^3$ with $2$-forms $\omega$, the Hodge Laplacian $-\Delta \omega$ acting on $2$-forms $\omega$ corresponds to  
$\operatorname{curl}\operatorname{curl}u- \nabla \operatorname{div}u$ and the absolute boundary conditions correspond to the boundary conditions $u \times \nu =0$ on $\partial\Omega$ in \eqref{el:maxwell:pb} together with  the boundary condition $\operatorname{div}u=0$ on $\partial\Omega$ (which in our problem is included in the second condition in \eqref{el:maxwell:pb}). On the other hand,  restricting the Hodge Laplacian to  exact $2$-forms corresponds to imposing the second condition $\operatorname{div}u=0$ in $\Omega$ in \eqref{el:maxwell:pb}. We conclude that  $\mu_1^{[2]'}=\lambda_1^{\Omega}$. Importantly, by the proof of \cite[Thm.~2.6]{guesavtrans} we have that $\mu_1^{[2]'}=\mu_1^{[2]}$, see also \cite[p.~1802]{savo} (we note that this result could also be deduced by the inequality between Maxwell and Helmholtz eigenvalues proved  in   \cite[Thm.~1.1]{roh} for domains that are not necessarily convex). 

Consider now the John ellipsoid of $\Omega$, i.e., the ellipsoid of maximal volume contained in $\Omega$ and denote by 
$D_1, D_2, D_3$ the lenghts of its principal axes, ordered by $D_1\geq D_2\geq D_3$. It is proved in \cite[Thm.~1.1]{savo} that 
$$
\frac{a}{D_2^2}\le  \mu_1^{[2]} \le \frac{a'}{D_2^2}
$$
for some positive constants $a,a'$ independent of $\Omega$. By the previous observations, it follows that 
$
aD_2^{-2}\le \lambda_1^{\Omega} \le a'D_2^{-2}. 
$
Accordingly, if we consider a cuboid $\Omega$ with sides $\ell_1, \ell_2, \ell_3 $ in the order $\ell_1\geq \ell_2\geq \ell_3 $, we have that $D_2=l_2$ hence 
$$
\frac{a}{\ell_2^2}\le \lambda_1^{\Omega} \le \frac{a'}{\ell_2^2}. 
$$
Keeping the volume of the cuboid $\Omega$ fixed and letting $\ell_2$ go either to $\infty$ or to $0$ we obtain that either 
$ \lambda_1^{\Omega}\to 0$ or $ \lambda_1^{\Omega}\to \infty$, respectively. In this way, we retrieve in an alternative asymptotic form what is explicitly written in  formulas  \eqref{zero} and $\eqref{infinito}$. Similarly, one can retrieve formula \eqref{infinitodue}.

The advantage of this approach lies in the fact that, in principle, one may use not only cuboids but also more general convex domains with suitable John ellipsoids. However, using cuboids and the corresponding  explicit formulas allows us to write  one-line proofs.
\end{Rem}

\begin{Rem} \label{ball}
The ball $B_k$ of surface area  $k$ has radius $R=\sqrt{k/(4 \pi)}$, hence
\begin{equation*}
\inf_{\substack{\Omega \text{\ \rm   cuboid}  \\ |\partial\Omega|=k }} \lambda_1^\Omega =
 \frac{ 4\pi^2}{k}<\lambda_1^{B_k} = \left(\frac{a'_{1,1}}{R} \right)^2 = \frac{4\pi  (a'_{1,1})^2}{k} <\frac{12 \pi^2}{k}=\lambda_1^{Q_k} ,
\end{equation*}
where $Q_k$ is the cube with surface area  $k$. Here  $a'_{1,1} \approx 2.7437 \pm 0.0001$ is the first positive zero of the derivative $\psi'_1$ of the Riccati-Bessel function $\psi'_1$ defined by $\psi_1(z)=zj_1(z)$ where $j_1$ is  the usual spherical Bessel function of the first kind, see e.g., \cite{lamzac21}. 
Thus, from the point of view of the minimization problem under perimeter constraint,  the ball is better than the cube but it is worse than other cuboids. 

Moreover,  if $\tilde B_k$ is a ball and $\tilde{Q}_k$ is a cube both of volume  $k$,  then the radius of the ball is $R=\sqrt[3]{\frac{3k}{4\pi}}$ hence
$$
\lambda_1^{\tilde B_k} = \left(\frac{a'_{1,1}}{R} \right)^2 =    (a'_{1,1})^2 \sqrt[3]{\frac{16\pi^2}{9k^2} }   < \frac{2 \pi^2}{\sqrt[3]{k^2}}=\lambda_1^{\tilde Q_k}. 
$$
Thus, from the point of view of the minimization problem under volume  constraint, the cube is worse than the ball. 
\end{Rem}

\section{Dumbbell domains}

In this section we consider the eigenvalue problem  \eqref{el:maxwell:pb}
 on a three dimensional dumbbell domain obtained as a product of a suitable  two dimensional dumbbell domain  as in Figure~\ref{figura:h} and an interval. Namely, for any $\delta , \eta >0 $ with $\delta <1$ and  $\eta <\delta$,  we consider the two dimensional dumbbell domain  in Figure \ref{figura1} defined as follows
$\omega_{\delta ,\eta}=G\cup S_{\delta, \eta}\cup G_{\delta}$ where 
$$G=(-1,0) \times \left(\frac{\left(-1+\eta \right) }{2},\frac{1+\eta }{2} \right),\ \ S_{\delta, \eta}=[0,\delta  ] \times (0,\eta )$$
and
$$G_{\delta} = (\delta , 2\delta ) \times \left(\frac{-\delta +\eta }{2},\frac{\delta +\eta }{2}\right).$$
Then we can  prove the following theorem that makes use  of Lemma~\ref{lemma:zeroneumann} below. 

\begin{Thm} Let $\beta>0$ be fixed. For any $h>0$ sufficiently big, let $\delta (h), \eta (h)$ be two positive real numbers depending on $h$ satisfying the following conditions:
$$
\delta (h)=o(h^{-2/\beta}),\ \ {\rm and}\ \ \eta (h)=O(\delta^{3+\beta} (h)),\ \ {\rm as}\ h\to +\infty .
$$ 
Moreover,  let  $\Omega(h)$ the three dimensional dumbbell domain defined by 
$$
\Omega_h := \omega_h\times (0,h),
$$
where 
$$
\omega_h :=L(h)\omega_{\delta (h),\eta (h)}
$$
and 
the positive factor $L(h)$ is defined by the condition
$
|\partial \Omega_h|=1. 
$
Then 
$$
\lim_{h\to +\infty } \lambda_1^{ \Omega_h}=0.
$$
\end{Thm}

\begin{proof}

\begin{figure}[b] 
\centering
\begin{tikzpicture} 
\draw (0,0) -- (4,0) -- (4,1.80);
\draw (4,2.20) -- (4,4) -- (0,4) -- (0,0);

\draw (4,1.80) -- (5.15,1.80);
\draw (4,2.20) -- (5.15,2.20);

\draw (5.15,1.80) -- (5.15,1.425) -- (6.3,1.425) -- (6.3,2.575) -- (5.15,2.575) -- (5.15,2.20);

\coordinate (a) at (0,0);
\coordinate (b) at (0,4);
\draw[<->] ([xshift=-0.3cm]a) -- ([xshift=-0.3cm]b) node[midway, left]{$\scriptstyle L(h)$};

\coordinate (c) at (4.1,1.80);
\coordinate (d) at (5.05,1.80);
\draw[<->] ([yshift=-0.2cm]c) -- ([yshift=-0.2cm]d) node[midway, below]{$\scriptstyle L(h)\delta(h)$};

\coordinate (e) at (6.3,1.425);
\coordinate (f) at (6.3,2.575);
\draw[<->] ([xshift=0.2cm]e) -- ([xshift=0.2cm]f) node[midway, right]{$\scriptstyle L(h) \delta(h)$};

\coordinate (g) at (4,1.82);
\coordinate (h) at (4,2.18);
\draw[<->] ([xshift=-0.2cm]g) -- ([xshift=-0.2cm]h) node[midway, left]{$\scriptstyle L(h) \eta(h)$};
\end{tikzpicture}
\caption{The base domain $\omega_h$} \label{figura:h}
\end{figure}
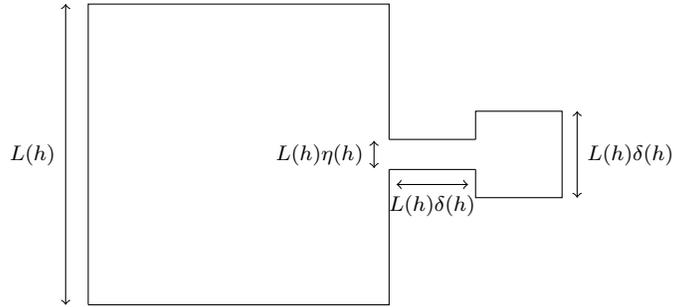

In the sequel, for the sake of simplicity, we shall often  omit  the dependence on $h$ of the quantities involved.

The surface area of $\partial \Omega_h$ is equal to 
\begin{equation*}
|\partial\Omega_h| = 2|\omega_h| + h  |\partial\omega_h| = 2L^2\left( 1+\delta(\delta+ \eta)\right) +2 h L ( 2 + 3 \delta -\eta). 
\end{equation*}
Imposing the constraint  $|\partial\Omega_h|=1$  is equivalent to   assuming that
\begin{equation*}
L=L(h) = \frac{1}{\sqrt{A^2+B}+A}
\end{equation*}
where
\begin{equation*}
A= h(2+3\delta-\eta) \qquad \text{ and } \qquad B = 2\left(1+\delta(\delta+\eta)\right) .
\end{equation*}
Note that $A\sim 2h$ and  $B\sim 2$ as $h\to +\infty$. Therefore $L \sim \frac{1}{4h}$ as $h\to +\infty$.
Note also   that $\eta <\delta <1$ for $h>0$ large enough.

Denoting by $\mu_{1,\mathcal{N}}^{U}$ the first positive eigenvalue of the Neumann Laplacian on a bounded Lipschitz domain $U$, by Lemma~\ref{lemma:zeroneumann} and our assumptions on $\delta$ and $\eta$ it follows  that
\begin{equation*}
\mu_{1,\mathcal{N}}^{\omega_h} = \frac{\mu_{1,\mathcal{N}}^{\omega_{\delta , \eta} }}{L^2} = \frac{O(\delta^{\beta}) }{L^2} = h^2 O(\delta^{\beta})=h^2o(h^{-2})\to 0\ \ {\rm as}\ \ h\to+\infty\, .\end{equation*}

We recall that for a product domain $\Omega = \omega \times I$ where  $\omega$ is a bounded Lipschitz domain in $\R^2$ and $I =(0,h)$ is an interval of length $h>0$, we have 
\begin{equation}  \label{1stmaxwdicotomy}
\lambda_1^\Omega = \min\bigg\{ \mu_{1,\mathcal{D}}^\omega,\  \mu_{1,\mathcal{N}}^\omega + \frac{\pi^2}{h^2}\bigg\},
\end{equation}
where $ \mu_{1,\mathcal{D}}^\omega$ is the first eigenvalue of the Dirichlet Laplacian on $\omega$, see \cite[Thm. 1.4]{cosdau}.
Clearly, since the area of $ \omega^h$ vanishes as $h\to +\infty$, we have that $  \mu_{1,\mathcal{D}}^{\omega_h}\to +\infty$, hence
\begin{equation*}
\lambda_1^{\Omega_h} = \mu_1^{\omega_h} + \frac{\pi^2}{h^2} \to 0
\end{equation*}
as $h\to +\infty$. 
\end{proof}

The following lemma is about a celebrated  example from \cite[Ch.VI \S 2]{couhil} that shows that the first Neumann eigenvalue on a certain dumbbell domain goes to zero as the channel of the dumbbell shrinks  fast enough. Since we need precise information on the decay rate, we include a proof for the convenience of the reader. 

\begin{Lem} \label{lemma:zeroneumann} Let   $\mu_{1,\mathcal{N}}^{\omega_{\delta, \eta}}$ be the first positive eigenvalue of the Neumann Laplacian on the two dimensional domain $\omega_{\delta , \eta }$ with $0<\eta<\delta<1$, see Figure \ref{figura1}.
If $\eta=o(\delta^3)$  then $\mu_{1,\mathcal{N}}^{\omega_{\delta, \eta}}\to 0$ as $\delta \to 0$. In particular, if  $\eta = O(\delta^{3 +\beta})$ with $\beta>0$, then
\begin{equation*}
 \mu_{1,\mathcal{N}}^{\omega_{\delta, \eta}} = O(\delta^\beta)
\end{equation*}
as $\delta\to 0$.
\end{Lem}

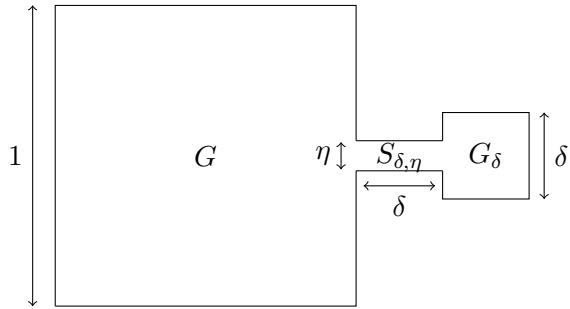
\begin{figure}[b] 
\centering
\begin{tikzpicture} 
\draw (0,0) -- (4,0) -- (4,1.80);
\draw (4,2.20) -- (4,4) -- (0,4) -- (0,0);

\draw (4,1.80) -- (5.15,1.80);
\draw (4,2.20) -- (5.15,2.20);

\draw (5.15,1.80) -- (5.15,1.425) -- (6.3,1.425) -- (6.3,2.575) -- (5.15,2.575) -- (5.15,2.20);

\coordinate (a) at (0,0);
\coordinate (b) at (0,4);
\draw[<->] ([xshift=-0.3cm]a) -- ([xshift=-0.3cm]b) node[midway, left]{1};

\coordinate (c) at (4.1,1.80);
\coordinate (d) at (5.05,1.80);
\draw[<->] ([yshift=-0.2cm]c) -- ([yshift=-0.2cm]d) node[midway, below]{$\delta$};

\coordinate (e) at (6.3,1.425);
\coordinate (f) at (6.3,2.575);
\draw[<->] ([xshift=0.2cm]e) -- ([xshift=0.2cm]f) node[midway, right]{$\delta$};

\coordinate (g) at (4,1.82);
\coordinate (h) at (4,2.18);
\draw[<->] ([xshift=-0.2cm]g) -- ([xshift=-0.2cm]h) node[midway, left]{$\eta$};

\node[] at (2,2)  {$G$};
\node[] at (4.575,1.975)  {$S_{\delta,\eta}$};
\node[] at (5.725,2)  {$G_\delta$};
\end{tikzpicture}

\caption{the domain $\omega_{\delta , \eta}$} \label{figura1}
\end{figure}

\begin{proof}
By the min-max characterization of the eigenvalues we have
\begin{equation*}
\mu_1^{\omega_{\delta, \eta}} = \inf_{\substack{u \in H^1(\omega_{\delta, \eta}) \\ \int_{\omega_{\delta, \eta} } u=0}} \frac{\int_{\omega_{\delta, \eta}} |\nabla u|^2}{\int_{\omega_{\delta, \eta}} u^2}.
\end{equation*}
Consider  the following function
\begin{equation*}
u(x,y):=
\begin{cases}
c, & (x,y) \in G,\\
c-\frac{1}{\delta}\left( \frac{1}{\delta} + c \right) x, & (x,y) \in S_{\delta, \eta},\\
-\frac{1}{\delta}, & (x,y) \in G_\delta, 
\end{cases}
\end{equation*}
where $c \in \mathbb{R}$ is a constant yet to be determined.
Note that  $u$ is Lipschitz continuous hence it belongs to $H^1(\omega_{\delta, \eta})$.
Since
\begin{equation*}
\begin{split}
\int_{\omega_{\delta, \eta}} u &= c - \delta + \eta \delta c - \frac{\eta}{\delta} \left( \frac{1}{\delta} +c \right) \int_0^\delta x dx =c - \delta + \frac{1}{2}\eta \delta  c -\frac{\eta}{2},
\end{split}
\end{equation*}
in order to guarantee that  $\int_{\omega_{\delta, \eta}   } u=0$, we set $c=c(\eta,\delta)=\frac{\eta+2\delta}{\eta \delta + 2}$.
Observe that $\lim_{\delta \to 0} c=0$. Moreover
\begin{equation*}
\nabla u (x,y) = 
\begin{cases}
(0,0), & (x,y) \in G,\\
(-\frac{1}{\delta}\left( \frac{1}{\delta} + c \right),0), & (x,y) \in S_{\delta, \eta},\\
(0,0), & (x,y) \in G_\delta.
\end{cases}
\end{equation*}
Therefore
\begin{equation*}
\int_{\omega_{\delta ,h}} |\nabla u|^2= \frac{\eta}{\delta}\left( \frac{1}{\delta} + c \right)^2 =  \frac{\eta}{\delta}  \left( \frac{1}{\delta} + \frac{\eta+2\delta}{\eta\delta +2} \right)^2 =O\left( \frac{\eta}{\delta^3}\right) \quad \text{ as } \delta \to 0,
\end{equation*}
while
\begin{equation*}
\begin{split}
\int_{\omega_{\delta, \eta}} u^2 &= c^2+1 + \eta \int_0^\delta \left( c-\frac{1}{\delta}\left( \frac{1}{\delta} + c \right) x \right)^2 dx 
\\
&=c^2+1 + \frac{\eta}{3\delta} (1-c \delta +c^2 \delta^2) \xrightarrow[\delta \to 0]{}1+ \frac{1}{3}\lim_{\delta \to 0} \frac{\eta}{\delta}.
\end{split}
\end{equation*}
Assuming $\eta=o(\delta^3)$ as $\delta \to 0$, we have 
\begin{equation*}
0<\mu_1^{\omega_{\delta ,\eta}} \leq \frac{\int_{\omega_{\delta, \eta}} |\nabla u|^2dx}{\int_{\omega_{\delta, \eta}} u^2dx} = O\left( \frac{\eta}{\delta^3}\right) \xrightarrow[\delta \to 0]{} 0.
\end{equation*}
Finally, choosing  $\eta=O( \delta^{3+\beta})$ with $\beta>0$ as $\delta \to 0$, we get that
$
\mu_1^{\omega_{\delta ,\eta} } =  O(\delta^\beta)
$
as $\delta \to 0$.
\end{proof}

\section*{Acknowledgments}
The authors are very thankful to Luigi Provenzano for suggesting \cite[Thm.~1.1]{savo} and the alternative proof discussed in Remark~\ref{savorem}. They are also very thankful to  Giovanni Franzina for pointing our reference \cite{guesav} and useful discussions. 

D.K.\ was supported by the EXPRO grant No.~20-17749X 
of the Czech Science Foundation.
P.D.L.  acknowledges the support  from the project ``Perturbation problems and asymptotics for elliptic differential equations: variational and potential theoretic methods'' funded by the MUR Progetti di Ricerca di Rilevante Interesse Nazionale (PRIN) Bando 2022 grant 2022SENJZ3.  
P.D.L is  member of the ``Gruppo Nazionale per l’Analisi Matematica, la Probabilit\`a e le loro Applicazioni'' (GNAMPA) of the ``Istituto Nazionale di Alta Matematica'' (INdAM). M.Z.   acknowledges  financial support by INdAM   through their program ``Borse di studio per l'estero''. 
The work of   M.Z. was partially realized during his stay at the  Czech Technical University in Prague with the support of the EXPRO grant No. 20-17749X of the Czech Science Foundation.
M.Z. also received support from the French government under the France 2030 investment plan, as part of the Initiative d'Excellence d'Aix-Marseille Université -- A*MIDEX  AMX-21-RID-012.

\end{document}